\documentclass[12pt]{amsart}
\usepackage{amssymb,amsmath,mathrsfs}
\usepackage[colorlinks=true,urlcolor=blue,
citecolor=red,linkcolor=blue,linktocpage,pdfpagelabels,
bookmarksnumbered,bookmarksopen]{hyperref}
\usepackage[active]{srcltx}
\usepackage{verbatim}
\usepackage{epsfig,graphicx,color,mathrsfs}
\usepackage{graphicx}
\usepackage{amsmath,amssymb,amsthm,amsfonts}
\usepackage{amssymb}
\usepackage[english]{babel}
\usepackage[left=2.7cm,right=2.7cm,top=3cm,bottom=3cm]{geometry}

\newtheorem{thm}{Theorem}[section]

\newtheorem{lem}[thm]{Lemma}

\newtheorem{defn}[thm]{Definition}

\newcommand{\R}{{\mathbb{R}}}

\renewcommand{\rho}{\varrho}
\renewcommand{\theta}{\vartheta}

\makeatletter
\@namedef{subjclassname@2020}{%
	\textup{2020} Mathematics Subject Classification}
\makeatother

\begin{document}


\parindent 0pc
\parskip 6pt
\overfullrule=0pt

\title[The moving plane method]{The moving plane method for  doubly singular elliptic equations involving a first order term}

\author{F. Esposito* and B. Sciunzi*}

\date{\today}

\address{* Dipartimento di Matematica e Informatica, UNICAL,
Ponte Pietro  Bucci 31B, 87036 Arcavacata di Rende, Cosenza, Italy.}

\email{sciunzi@mat.unical.it}
\email{francesco.esposito@unical.it}



\thanks{F. Esposito and B. Sciunzi were partially supported by PRIN project  2017JPCAPN (Italy): {\em Qualitative and quantitative aspects of nonlinear PDEs}.}

\maketitle

\date{\today}

\begin{abstract}
In this paper we deal with positive singular solutions to semilinear elliptic problems involving a first order term and a singular nonlinearity. Exploiting a fine adaptation of the well-known moving plane method of Alexandrov-Serrin and a careful choice of the cutoff functions, we deduce symmetry and monotonicity properties of the solutions.
\end{abstract}

\section{Introduction}
The aim of this work is to investigate qualitative properties of singular solutions of semilinear elliptic problems involving a first order term and a singular nonlinearity. In order to achieve our purpose, we consider the following problem under zero Dirchlet boundary condition
\begin{equation} \label{problem}
\begin{cases}
\displaystyle -\Delta u + |\nabla u|^2 \,= \frac{1}{u^\gamma} + f(x,u)& \text{in}\quad\Omega\setminus \Gamma  \\
u> 0 &  \text{in}\quad\Omega\setminus \Gamma  \\
u=0 &  \text{on}\quad\partial \Omega\,
\end{cases}
\end{equation}
where $\gamma >1$, $\Omega$ is a bounded smooth domain of $\mathbb{R}^n$ with $n\geq 2$ and the nonlinearity $f$ is assumed to be a uniformly locally Lipschitz continuous function far from the singular set $\Gamma$ (see Definition \ref{def:nonlinAssump}). We obtain our result by adapting the celebrated moving plane technique introduced by Alexandrov  in his
study of surfaces of constant mean curvature (see \cite{A}) and by Serrin in the context of overdetermined problems for partial differential equations (see \cite{serrin}). Moreover, we take into account the variant developed in the seminal papers of Gidas-Ni-Nirenberg \cite{GNN} and Berestycki-Nirenberg \cite{BN}. We want to stress that this method is very powerful and can be adapted to several problems, hence, for this reason, we refer the reader to \cite{BN1, BN2, ChenLin, CirRon, lucio, DamPac1, DamPac2, DS1, Dan2, EST, Gidas, Li, Monticelli, Dino} for some related works regarding such a technique in bounded domains. This procedure can be performed in general domains providing partial monotonicity results near the boundary and symmetry when the domain is convex and symmetric. To this end, for simplicity of exposition, we assume directly in all the paper that $\Omega$ is a convex domain which is symmetric with respect to the hyperplane $\Pi_0:=\{x_1=0\}$. The solution has a possible singularity on the critical set $\Gamma\subset \Omega$. 

The literature regarding singular solutions is really wide and we have tried to outline all the appropriate references. We apologize if we do not find all the related works. First of all, we suggest the reading of the book of V\'eron \cite{veron} and the references therein. Moreover, for the adaptation of the moving plane in the case of a point singularity we refer to the seminal papers of Terracini \cite{T} and of Caffarelli-Li-Nirenberg \cite{CLN2}.\\
The case $\Gamma=\emptyset$ with singular nonlinearities, without  the first order term, has been widely investigated
in the literature. We refer the readers to the pioneering work of Crandall-Rabinowitz-Tartar \cite{crandall} and to
\cite{boccardo, CES, gras1, luigi, nodt, ES, lazer, oliva, stuart} for other contributions for semilinear elliptic equations involving singular nonlinearity. In
particular, by \cite{lazer}, it is known that
 solutions generally have no $H^1$-regularity up to the boundary. Therefore, having in mind this work,  the natural assumption in our paper is
$$u \in H_{loc}^1(\Omega \setminus \Gamma) \cap
C(\overline{\Omega} \setminus \Gamma),$$
and thus the equation is
understood in the weak distributional meaning, i.e.
\begin{equation}\label{debil1}
\int_\Omega \langle \nabla u ,\nabla \varphi \rangle\,dx + \int_\Omega |\nabla u|^2 \varphi \, dx\,=\, \int_\Omega  \frac{\varphi}{u^\gamma} \, dx + \int_\Omega f(x,u)\varphi\,dx,\qquad\text{for all } \varphi\in C^{1}_c(\Omega\setminus\Gamma)\,.
\end{equation}

We point out that, since the right hand side of \eqref{problem} is locally bounded, by standard elliptic regularity theory, it follows that $u \in C_{loc}^{1,\alpha} (\Omega \setminus \Gamma)$, where $0 < \alpha < 1$.

Now we are ready to state the main result of this work:

\begin{thm}\label{main}
Suppose that $\Omega$ is a convex domain which is symmetric with respect to the hyperplane $\Pi_0=\{x_1=0\}$ and let $u\in H^1_{loc}(\Omega\setminus\Gamma)\cap
C(\overline\Omega\setminus\Gamma)$ be a solution to \eqref{problem}.
Assume that $f$ fulfills $(\mathbf{hp}_f)$ (see Definition \ref{djkfhsdjk}), and that $\Gamma$ is a point if $n=2$, while
 it is closed and such that
$$\underset{\R^n}{\operatorname{Cap}_2}(\Gamma)=0,$$
if $n\geq 3$.
\noindent Then, if $\Gamma\subset\Pi_0$, it follows that   $u$
is symmetric with respect to the hyperplane $\Pi_0$ and
increasing in the $x_1$-direction in $\Omega\cap\{x_1<0\}$.
Furthermore,
\[\frac{\partial u}{\partial x_1}>0\qquad \text{in}\quad \Omega\cap\{x_1<0\}\,.
\]
\end{thm}

In the proof of Theorem \ref{main} we develop a careful adaptation of the moving plane method that is quite recent and it was introduced in \cite{EFS, Dino} in order to deal with singular solutions of semilinear elliptic problems driven by the classical
Laplacian operator. In our work, the presence of the first order term $|\nabla u|^2$ and of the singular nonlinearity $u^{-\gamma}$ pushed us to adapt the procedure proposed in \cite{EFS, Dino} with a careful choice of the cut-off functions. To the best of our knowledge, Theorem \ref{main} is new and extends some of the results contained in \cite{EFS,Dino} to the case involving first order terms. This technique is so powerful and flexible that  covers also the following cases: unbounded sets  \cite{EFS,Dino}, the $p$–Laplacian operator \cite{EMS, MontoroHarnack}, double phase operators \cite{BEV}, cooperative elliptic systems \cite{BVV0, BVV, esposito}, the fractional Laplacian \cite{mps} and mixed local–nonlocal elliptic operators \cite{BDVV}.

\section{Notations and preliminary results} \label{notations}

The aim of this section is to fix once and for all the relevant
notations used throughout the paper, and to present some auxiliary results, proved in \cite{EFS}, which shall be key ingredients for the proof of Theorem \ref{main}.\\
As pointed out in the introduction, in all the paper we suppose that the nonlinearity $f$ is uniformly locally Lipschitz continuous far from the singular set $\Gamma$. More precisely, we state the following:
\begin{defn}[$\mathbf{hp}_f$] \label{def:nonlinAssump}
	\label{djkfhsdjk} We say that $f$ fulfills the condition $(\mathbf{hp}_f)$
	if $f: {\overline{ \Omega }}\setminus \Gamma \times (0,+\infty )
	\rightarrow {\mathbb{R}}$ is a continuous function such that for $0	\, \leq s,t\leq \Lambda$ and for any compact set $\mathcal{K}\subset
	{\overline{\Omega }} \setminus \Gamma $, it holds
	\begin{equation}
		\nonumber |f(x,s)-f(x,t)|\leq \mathcal{L}_f(\Lambda, \mathcal{K})|s-t|, \qquad \text{for all}\quad
		x\in \mathcal{K}\,,
	\end{equation}
	where $\mathcal{L}_f(\Lambda, \mathcal{K})$ is a positive constant depending on $\Lambda$ and $\mathcal{K}$.
	Furthermore, $f(\cdot ,s)$ is non-decreasing in the $x_{1}$-direction
	in $\Omega \cap \{x_{1}<0\}$ and symmetric with respect to the
	hyperplane $\Pi_0$.
\end{defn}

Let $\Gamma \subset \Omega \subset \R^n$ be as in the statement
of Theorem \ref{main}, and let
$u \in H_{loc}^1(\Omega \setminus \Gamma) \cap
C(\overline{\Omega} \setminus \Gamma)$ be a solution
of \eqref{problem}. For any fixed $\lambda\in\R$, we denote by $R_\lambda$
the reflection trough the hy\-per\-pla\-ne $\Pi_\lambda := \{x_1 = \lambda\}$,
that is,
\begin{equation} \label{eq.defRlambda}
	R_\lambda(x) = x_\lambda := (2\lambda-x_1,x_2,\ldots,x_N)
	\qquad (\text{for all $x\in\R^n$}).
\end{equation}
Hence, we define the function
\begin{equation} \label{eq.defulambda}
	u_\lambda  := u \circ R_\lambda.
\end{equation}
We point out that, since $u$ solves \eqref{problem}, one has
\begin{itemize}
	\item $u_\lambda\in  H^1_{loc}(R_\lambda(\Omega\setminus\Gamma)) \cap C(R_\lambda(\overline{\Omega}\setminus\Gamma)) $;
	\item $u_\lambda > 0$ in $R_\lambda(\Omega\setminus\Gamma)$ and 
	$u_\lambda \equiv 0$ on $R_\lambda(\partial\Omega\setminus\Gamma)$;
	\item for every test function $\varphi\in C^1_c(R_\lambda(\Omega\setminus\Gamma))$
	one has
\begin{equation}\label{debil2}
	\int_{R_\lambda(\Omega)} \langle \nabla u_\lambda, \nabla
	\varphi \rangle \,dx \, + \int_{R_\lambda(\Omega)} |\nabla u|^2 \varphi \, dx\,=\, \int_{R_\lambda(\Omega)}  \frac{\varphi}{u^\gamma} \, dx \, + \, \int_{R_\lambda(\Omega)} f(x_\lambda,u)\varphi\,dx\,
\end{equation}
and we also observe that, for any  $a<\lambda<0$, the function $w_\lambda\,:=\,u-u_\lambda$ satisfies  $0 \le w_\lambda^+ \le u $ a.e. on  ${\Omega_{\lambda}}$ and so $w_\lambda^+ \in L^\infty(\Omega_{\lambda})$, since  $u \in C(\overline{\Omega}_{\lambda})$.
\end{itemize}
To proceed further, we define
\begin{equation} \label{eq.defaOmega}
	\mathbf{a} = \mathbf{a}_\Omega := \inf_{x\in\Omega}x_1
\end{equation}
and we observe that since $\Omega$ is bounded and symmetric with respect to
the $x_1$-direction,
we certainly have $-\infty < \mathbf{a} < 0$.
Hence, for every $\lambda\in(\mathbf{a},0)$ we can set
\begin{equation} \label{eq.defOmegalambda}
	\Omega_\lambda := \{x\in\Omega:\,x_1<\lambda\}.
\end{equation}
Notice that since
$\Omega$ is convex in the $x_1$-direction, then
\begin{equation} \label{eq.inclusionOmegalambda}
	\Omega_\lambda\subseteq R_\lambda(\Omega)\cap \Omega.
\end{equation}
Finally, for every $\lambda\in(\mathbf{a},0)$, we define the function
$$w_\lambda(x) := (u-u_\lambda)(x), \qquad \text{for
	$x\in (\overline{\Omega}\setminus\Gamma)\cap R_\lambda(\overline{\Omega}\setminus\Gamma)$}.$$
Taking into account \eqref{eq.inclusionOmegalambda}, we deduce that
$w_\lambda$ is well-defined on
$\overline{\Omega}_\lambda\setminus R_\lambda(\Gamma)$.

Now we say that it can be easily shown that if
$\underset{\R^n}{\operatorname{Cap}_2}(\Gamma)=0$, then
$\underset{\R^n}{\operatorname{Cap}_2}(R_{\lambda}(\Gamma))=0$.
Another consequence of this fact is that
$\underset{B^{\lambda}_{\epsilon}}{\operatorname{Cap}_2}(R_{\lambda}(\Gamma))=0$
for any open neighborhood $\mathcal{B}^{\lambda}_{\epsilon}$ of
$R_{\lambda}(\Gamma)$.
Indeed, recalling that $\Gamma$ is a point if $n=2$, while
$\Gamma$ is closed with
$\underset{\R^n}{\operatorname{Cap}_2}(\Gamma)=0$
if $n\geq 3$ by assumption, it follows that
$$\underset{\mathcal{B}^{\lambda}_{\epsilon}}{\operatorname{Cap}_2}(R_{\lambda}(\Gamma)):=
\inf \left\{ \int_{\mathcal{B}^{\lambda}_{\epsilon}} |\nabla \varphi
|^2 dx < + \infty \; : \; \varphi \geq 1 \ \text{in} \
\mathcal{B}^{\lambda}_{\delta}, \; \varphi \in \
C^{\infty}_c(\mathcal{B}^{\lambda}_{\epsilon}) \right\}=0,$$
\noindent for some neighborhood $\mathcal{B}^{\lambda}_{\delta}
\subset \mathcal{B}^{\lambda}_{\varepsilon}$ of
$R_{\lambda}(\Gamma)$. As a consequence (see \cite{evans} for more details) there exists
$\varphi_{\varepsilon} \in
C^{\infty}_c(\mathcal{B}^{\lambda}_{\varepsilon})$ such that
\begin{equation} \label{eq.choicephieps}
	\text{$\varphi_\varepsilon\geq 1$ on $R_\lambda(\Gamma)$}\qquad
	\text{and}\qquad
	\int_{\mathcal{B}^{\lambda}_{\varepsilon}} |\nabla
	\varphi_{\varepsilon}|^2\,dx < \varepsilon.
\end{equation}
We then consider the Lipschitz function
$$T(s)= \begin{cases}
	1 &  \text{if}\quad s \le 0  \\
	-2s + 1 & \text{if}\quad 0 \le s \le \frac{1}{2} \\
	0 &  \text{if}\quad s \ge \frac{1}{2}  \\
\end{cases}$$
and we define 
\begin{equation} \label{test1}
	\psi_{\varepsilon} :=  T \circ \varphi_{\varepsilon}.
\end{equation}
In view of \eqref{eq.choicephieps}, and taking into account
the  definition of $T$, it is not difficult to show that
$\psi_\varepsilon \equiv 1$ on $\R^n\setminus \mathcal{B}^\lambda_\varepsilon$ and $\psi_\varepsilon\equiv 0$ on the neighborhood $\mathcal{B}^{\lambda}_{\delta}$ of $R_\lambda(\Gamma)$, $0\leq \psi_\varepsilon\leq 1$ on $\R^n$, $\psi_\varepsilon$ is Lipschitz-continuous in $\R^n$. Moreover, there exists a constant $\Upsilon_1  > 0$, independent of $\varepsilon$, such that
$$\int_{\R^n}|\nabla \psi_{\varepsilon}|^2\,  dx = \int_{\mathcal{B}^\lambda_\varepsilon \setminus \mathcal{B}^\lambda_\delta}|\nabla \psi_{\varepsilon}|^2\,  dx \leq \Upsilon_1 \int_{\mathcal{B}^\lambda_\varepsilon}|\nabla \varphi_{\varepsilon}|^2\,  dx \leq  \Upsilon_1 \varepsilon.$$

\medskip

Now we  set $\gamma_\lambda:= \partial \Omega \cap T_{\lambda}$. As pointed out in \cite{EFS}, recalling that $\Omega$ is convex, it is easy to  deduce that $\gamma_\lambda$ is made of two points in dimension two. If else $n\geq3$ then it follows that
$\gamma_\lambda$ is a smooth manifold of dimension $n-2$.
Moreover, it can be shown that this implies 
$\underset{\R^n}{\operatorname{Cap}_2}(\gamma_\lambda)=0$, see  e.g.
\cite{evans}.
So, as before, $\underset{\mathcal{I}^\lambda_{\tau}}{\operatorname{Cap}_2}(\gamma_\lambda)=0$
for any open neighborhood of $\gamma_\lambda$ and then
there exists $\varphi_{\tau} \in C^{\infty}_c (\mathcal{I}^\lambda_{\tau})$ such that
$\varphi_{\tau} \geq 1$ in a neighborhood $\mathcal{I}^\lambda_{\sigma}$ with
 $\gamma_\lambda\subset\mathcal{I}^\lambda_{\sigma} \subset
\mathcal{I}^\lambda_{\tau}$. As above, we set
\begin{equation}\label{test2}
	\phi_{\tau} :=T \circ \varphi_{\tau},
\end{equation}
where we have extended  $\varphi_{\tau}$ by zero outside $\mathcal{I}^\lambda_{\tau}$. Then $\phi_\tau \equiv 1$ on $\R^n\setminus \mathcal{B}^\lambda_\varepsilon$ and $\phi_\tau \equiv 0$ on the neighborhood $\mathcal{I}^\lambda_{\sigma}$ of $R_\lambda(\Gamma)$, $0\leq \phi_\tau \leq 1$ on $\R^n$, $\phi_\tau$ is Lipschitz-continuous in $\R^n$; moreover, there exists a constant $\Upsilon_2  > 0$, independent of $\tau$, such that
$$\int_{\R^n} |\nabla \phi_\tau|^2 dx = \int_{\mathcal{I}^\lambda_{\tau} \setminus \mathcal{I}^\lambda_{\sigma}} |\nabla \phi_\tau|^2 dx \leq \Upsilon_2  \int_{\mathcal{I}^\lambda_{\tau}} |\nabla \varphi_\tau|^2 dx  \leq  \Upsilon_2 \tau.$$

To proceed further, in the same spirit of \cite{EFS}, we need of the following result:

\begin{lem}\label{leaiuto} 
	Let $\lambda \in (\mathbf{a},0)$ be such that $R_\lambda (\Gamma) \cap {\overline {\Omega}} = \emptyset$ and consider the functions
	$$\varphi_{1,\lambda}\,:= \begin{cases}\, e^{-u} w_\lambda^+ \phi_{\tau}^2  & \text{in}\quad\Omega_\lambda \\
		0 &  \text{in}\quad \R^n\setminus \Omega_\lambda
	\end{cases} \quad \text{and} \quad \varphi_{2,\lambda}\,:= \begin{cases}\, e^{-u_\lambda} w_\lambda^+ \phi_{\tau}^2  & \text{in}\quad\Omega_\lambda \\
	0 &  \text{in}\quad \R^n\setminus \Omega_\lambda,
\end{cases}$$
	where $\phi_{\tau}$ is as in \eqref{test2}. Then $\varphi_{1,\lambda}, \, \varphi_{2,\lambda} \in C^{0,1}_c (\Omega) \cap C^{0,1}_c (R_{\lambda}(\Omega))$ have compact support contained in $(\Omega \setminus \Gamma) \cap (R_{\lambda}(\Omega) \setminus R_{\lambda}(\Gamma)) \cap \{ x_1 \le \lambda \}$ and
	\begin{equation}\label{gradvarphi1}
		\begin{split}	
		\nabla \varphi_{1,\lambda} &= e^{-u} \left[-w_\lambda^+ \phi_\tau^2 \nabla u + \phi_\tau^2 \nabla w_{i,\lambda}^+  + 2 \phi_\tau
		w_{i,\lambda}^+ \nabla \phi_\tau \right]\quad \ \ {\text {a.e. on }} \, \, \Omega \cup R_{\lambda}(\Omega),\\	
		\nabla \varphi_{2,\lambda} &= e^{-u_\lambda} \left[ -w_\lambda^+ \phi_\tau^2 \nabla u_\lambda + \phi_\tau^2 \nabla w_{i,\lambda}^+  + 2 \phi_\tau
		w_{i,\lambda}^+ \nabla \phi_\tau\right] \quad {\text {a.e. on }} \, \, \Omega \cup 		R_{\lambda}(\Omega).
	    \end{split}
    \end{equation}

	If $\lambda \in (\mathbf{a},0)$ is such that $R_\lambda (\Gamma) \cap {\overline {\Omega}} \neq \emptyset$, we achieve the same conclusions for the functions
	$$\varphi_{1,\lambda}\,:= \begin{cases}\, e^{-u} w_\lambda^+ \psi_\varepsilon^2 \phi_{\tau}^2  & \text{in}\quad\Omega_\lambda \\
	0 &  \text{in}\quad \R^n\setminus \Omega_\lambda
	\end{cases} \quad \text{and} \quad \varphi_{2,\lambda}\,:= \begin{cases}\, e^{-u_\lambda} w_\lambda^+ \psi_\varepsilon^2 \phi_{\tau}^2  & \text{in}\quad\Omega_\lambda \\
	0 &  \text{in}\quad \R^n\setminus \Omega_\lambda,
	\end{cases}$$
	where $\psi_\varepsilon$ is defined as in \eqref{test1} and $\phi_{\tau}$
	as in \eqref{test2}. Furthermore, 
	\begin{equation}\label{gradvarphi2}
		\begin{split}	
			\nabla \varphi_{1,\lambda} &= e^{-u} \left[-w_\lambda^+ \psi_\varepsilon^2 \phi_\tau^2 \nabla u + \psi_\varepsilon^2 \phi_\tau^2 \nabla w_{\lambda}^+  + 2 w_{\lambda}^+ \psi_\varepsilon \phi_\tau(  \phi_\tau \nabla  \psi_\varepsilon + \psi_\varepsilon \nabla \phi_\tau) \right],\\	
			\nabla \varphi_{2,\lambda} &= e^{-u_\lambda} \left[-w_\lambda^+ \psi_\varepsilon^2 \phi_\tau^2 \nabla u_\lambda + \psi_\varepsilon^2 \phi_\tau^2 \nabla w_{\lambda}^+  + 2 w_{\lambda}^+ \psi_\varepsilon \phi_\tau(  \phi_\tau \nabla  \psi_\varepsilon + \psi_\varepsilon \nabla \phi_\tau) \right],
		\end{split}
	\end{equation}
a.e.~on $\Omega \cup R_{\lambda}(\Omega)$.
\end{lem}

We omit the proof of this lemma, since it repeats verbatim the arguments given in the proof of Lemma 3.1 of \cite{EFS}.

\section{The moving plane method for singular solutions}\label{mainproofsec}

In this section we  adapt the moving plane technique in the same spirit of \cite{EFS} by a careful choice of the cut-off functions defined in Lemma \ref{leaiuto}. The first step needed to prove Theorem \ref{main} is to show that $w_\lambda^+ \in H^1_0(\Omega_\lambda)$. In particular, we give the following:
\begin{lem}\label{leaiuto2}
Under the assumptions of Theorem \ref{main}, let $\mathbf{a}<\lambda<0$. Then 
\begin{equation}\label{eq:H10test}
\int_{\Omega_\lambda}|\nabla w_\lambda^+|^2\,dx\leq \mathbf{C}(f,{\vert \Omega \vert}, \|u\|_{L^\infty(\Omega_\lambda)}),
\end{equation}
{where $\vert \Omega \vert$ denotes the $n-$dimensional Lebesgue measure of $ \Omega$. Consequently, $ w_\lambda^+\in H^1_0(\Omega_\lambda)$.}
\end{lem}

\begin{proof}

For $\psi_\varepsilon$ as in \eqref{test1} and $\phi_{\tau}$
as in \eqref{test2}, we consider the functions $\varphi_{1,\lambda}$ and $\varphi_{2,\lambda}$ defined in Lemma \ref{leaiuto}. According to the properties of $\varphi_{1,\lambda}$ and $\varphi_{2,\lambda}$ stated in Lemma \ref{leaiuto}, and a standard density argument, we can use $\varphi_{1,\lambda}$  as test function in \eqref{debil1} and $\varphi_{2,\lambda}$ in \eqref{debil2} so that we get
\begin{equation}\nonumber
	\begin{split}
	-\int_{\Omega_{\lambda}} e^{-u}  |\nabla u|^2 w_\lambda^+ \psi_\varepsilon^2 \phi_{\tau}^2 \, dx &+ \int_{\Omega_{\lambda}} e^{-u} \langle \nabla u, \nabla w_\lambda^+ \rangle  \psi_\varepsilon^2 \phi_{\tau}^2 \, dx +  2 \int_{\Omega_{\lambda}} e^{-u} \langle \nabla u, \nabla \psi_\varepsilon \rangle   w_\lambda^+ \psi_\varepsilon \phi_{\tau}^2 \, dx\\
	&+ 2 \int_{\Omega_{\lambda}} e^{-u} \langle \nabla u, \nabla \phi_\tau \rangle   w_\lambda^+ \psi_\varepsilon^2 \phi_{\tau} \, dx + \int_{\Omega_{\lambda}} e^{-u}  |\nabla u|^2 w_\lambda^+ \psi_\varepsilon^2 \phi_{\tau}^2 \, dx\\
	=& \int_{\Omega_{\lambda}} \frac{e^{-u}}{u^\gamma}   w_\lambda^+ \psi_\varepsilon^2 \phi_{\tau}^2 \, dx + \int_{\Omega_{\lambda}} e^{-u} f(x,u)   w_\lambda^+ \psi_\varepsilon^2 \phi_{\tau}^2 \, dx.
	\end{split}
\end{equation}
Hence, we have
\begin{equation}\label{eq:substtest1}
	\begin{split}
		\int_{\Omega_{\lambda}} e^{-u} \langle \nabla u, \nabla w_\lambda^+ \rangle  \psi_\varepsilon^2 \phi_{\tau}^2 \, dx &+  2 \int_{\Omega_{\lambda}} e^{-u} \langle \nabla u, \nabla \psi_\varepsilon \rangle   w_\lambda^+ \psi_\varepsilon \phi_{\tau}^2 \, dx \\
		&+ 2 \int_{\Omega_{\lambda}} e^{-u} \langle \nabla u, \nabla \phi_\tau \rangle   w_\lambda^+ \psi_\varepsilon^2 \phi_{\tau} \, dx \\
		=& \int_{\Omega_{\lambda}} \frac{e^{-u}}{u^\gamma}   w_\lambda^+ \psi_\varepsilon^2 \phi_{\tau}^2 \, dx + \int_{\Omega_{\lambda}} e^{-u} f(x,u)   w_\lambda^+ \psi_\varepsilon^2 \phi_{\tau}^2 \, dx.
	\end{split}
\end{equation}
Now, as remarked above, by Lemma \ref{leaiuto} and a standard density argument, we can use $\varphi_{2,\lambda}$ in \eqref{debil2} so that we obtain
\begin{equation}\nonumber
	\begin{split}
		-\int_{\Omega_{\lambda}} e^{-u_\lambda}  |\nabla u_\lambda|^2 w_\lambda^+ \psi_\varepsilon^2 \phi_{\tau}^2 \, dx &+ \int_{\Omega_{\lambda}} e^{-u_\lambda} \langle \nabla u_\lambda, \nabla w_\lambda^+ \rangle  \psi_\varepsilon^2 \phi_{\tau}^2 \, dx \\
		&+  2 \int_{\Omega_{\lambda}} e^{-u_\lambda} \langle \nabla u_\lambda, \nabla \psi_\varepsilon \rangle   w_\lambda^+ \psi_\varepsilon \phi_{\tau}^2 \, dx\\
		&+ 2 \int_{\Omega_{\lambda}} e^{-u_\lambda} \langle \nabla u_\lambda, \nabla \phi_\tau \rangle   w_\lambda^+ \psi_\varepsilon^2 \phi_{\tau} \, dx \\
		&+ \int_{\Omega_{\lambda}} e^{-u_\lambda}  |\nabla u_\lambda|^2 w_\lambda^+ \psi_\varepsilon^2 \phi_{\tau}^2 \, dx\\
		=& \int_{\Omega_{\lambda}} \frac{e^{-u_\lambda}}{u_\lambda^\gamma}   w_\lambda^+ \psi_\varepsilon^2 \phi_{\tau}^2 \, dx + \int_{\Omega_{\lambda}} e^{-u_\lambda} f(x_\lambda,u_\lambda)   w_\lambda^+ \psi_\varepsilon^2 \phi_{\tau}^2 \, dx.
	\end{split}
\end{equation}
Hence, we have
\begin{equation}\label{eq:substtest2}
\begin{split}
	\int_{\Omega_{\lambda}} e^{-u_\lambda} \langle \nabla u_\lambda, \nabla w_\lambda^+ \rangle  \psi_\varepsilon^2 \phi_{\tau}^2 \, dx &+  2 \int_{\Omega_{\lambda}} e^{-u_\lambda} \langle \nabla u_\lambda, \nabla \psi_\varepsilon \rangle   w_\lambda^+ \psi_\varepsilon \phi_{\tau}^2 \, dx\\
	& + 2 \int_{\Omega_{\lambda}} e^{-u_\lambda} \langle \nabla u_\lambda, \nabla \phi_\tau \rangle   w_\lambda^+ \psi_\varepsilon^2 \phi_{\tau} \, dx \\
	=& \int_{\Omega_{\lambda}} \frac{e^{-u_\lambda}}{u_\lambda^\gamma}   w_\lambda^+ \psi_\varepsilon^2 \phi_{\tau}^2 \, dx + \int_{\Omega_{\lambda}} e^{-u_\lambda} f(x_\lambda,u_\lambda)   w_\lambda^+ \psi_\varepsilon^2 \phi_{\tau}^2 \, dx.
\end{split}
\end{equation}
Subtracting \eqref{eq:substtest1} and \eqref{eq:substtest2} we obtain
\begin{equation}\label{eq:esitmatesubtracting1}
\begin{split}
	\int_{\Omega_{\lambda}} &e^{-u} \langle \nabla u, \nabla w_\lambda^+ \rangle  \psi_\varepsilon^2 \phi_{\tau}^2 \, dx  - \int_{\Omega_{\lambda}} e^{-u_\lambda} \langle \nabla u_\lambda, \nabla w_\lambda^+ \rangle  \psi_\varepsilon^2 \phi_{\tau}^2 \, dx\\
	=&-  2 \int_{\Omega_{\lambda}} e^{-u} \langle \nabla u, \nabla \psi_\varepsilon \rangle   w_\lambda^+ \psi_\varepsilon \phi_{\tau}^2 \, dx +   2 \int_{\Omega_{\lambda}} e^{-u_\lambda} \langle \nabla u_\lambda, \nabla \psi_\varepsilon \rangle   w_\lambda^+ \psi_\varepsilon \phi_{\tau}^2 \, dx\\
	&-2 \int_{\Omega_{\lambda}} e^{-u} \langle \nabla u, \nabla \phi_\tau \rangle   w_\lambda^+ \psi_\varepsilon^2 \phi_{\tau} \, dx + 2 \int_{\Omega_{\lambda}} e^{-u_\lambda} \langle \nabla u_\lambda, \nabla \phi_\tau \rangle   w_\lambda^+ \psi_\varepsilon^2 \phi_{\tau} \, dx\\
	&+\int_{\Omega_{\lambda}} \left(\frac{e^{-u}}{u^\gamma} - \frac{e^{-u_\lambda}}{u_\lambda^\gamma}\right)  w_\lambda^+ \psi_\varepsilon^2 \phi_{\tau}^2 \, dx + \int_{\Omega_{\lambda}} (e^{-u}f(x,u)-e^{-u_\lambda} f(x_\lambda,u_\lambda))   w_\lambda^+ \psi_\varepsilon^2 \phi_{\tau}^2 \, dx.
\end{split}
\end{equation}
Now, adding to both side of \eqref{eq:esitmatesubtracting1} the quantity $\displaystyle \int_{\Omega_{\lambda}} e^{-u_\lambda} \langle \nabla u, \nabla w_\lambda^+ \rangle  \psi_\varepsilon^2 \phi_{\tau}^2 \, dx$, adding and subtracting to the right hand side of \eqref{eq:esitmatesubtracting1} the quantities $\ \displaystyle 2 \int_{\Omega_{\lambda}} e^{-u_\lambda} \langle \nabla u, \nabla \psi_\varepsilon \rangle  w_\lambda^+ \psi_\varepsilon \phi_{\tau}^2 \, dx\,$ and  $\displaystyle 2 \int_{\Omega_{\lambda}} e^{-u_\lambda} \langle \nabla u, \nabla  \phi_{\tau} \rangle  w_\lambda^+ \psi_\varepsilon^2 \phi_{\tau} \, dx,$ applying the Schwarz inequality and using that the function $\displaystyle t \mapsto \frac{e^{-t}}{t^\gamma}$ is decreasing, we obtain 
\begin{equation}\label{stima1}
\begin{split}
\int_{\Omega_\lambda} e^{-u_\lambda}|\nabla w_\lambda^+|^2\psi_\varepsilon^2
\phi_{\tau}^2 \,dx \leq& \int_{\Omega_{\lambda}} |e^{-u} - e^{-u_\lambda}| \cdot |\nabla u| \cdot |\nabla w_\lambda^+|  \ \psi_\varepsilon^2 \phi_{\tau}^2 \, dx\\
&+ 2 \int_{\Omega_{\lambda}} e^{-u_\lambda} |\nabla w_\lambda^+| \cdot  |\nabla \psi_\varepsilon| \ w_\lambda^+ \psi_\varepsilon \phi_{\tau}^2 \, dx\\
&+2 \int_{\Omega_{\lambda}} e^{-u_\lambda} |\nabla w_\lambda^+| \cdot   |\nabla \phi_\tau| \   w_\lambda^+ \psi_\varepsilon^2 \phi_{\tau} \, dx\\
&+ 2 \int_{\Omega_{\lambda}} |e^{-u} - e^{-u_\lambda}| \cdot  |\nabla u| \cdot  |\nabla \psi_\varepsilon| \ w_\lambda^+ \psi_\varepsilon \phi_{\tau}^2 \, dx\\
&+ 2 \int_{\Omega_{\lambda}} |e^{-u} - e^{-u_\lambda}| \cdot  |\nabla u| \cdot  |\nabla \phi_\tau| \ w_\lambda^+ \psi_\varepsilon^2 \phi_{\tau} \, dx\\
&+ \int_{\Omega_{\lambda}} (e^{-u}f(x,u)-e^{-u_\lambda} f(x_\lambda,u_\lambda)) \  w_\lambda^+ \psi_\varepsilon^2 \phi_{\tau}^2 \, dx \\
=&:\mathcal{I}_1+\mathcal{I}_2+\mathcal{I}_3+\mathcal{I}_4+\mathcal{I}_5+\mathcal{I}_6.
\end{split}
\end{equation}
Our aim is to estimate the right hand side of \eqref{stima1}. We start with the first term:

\begin{equation}\label{I_1}
	\begin{split}
		\mathcal{I}_1&=\int_{\Omega_{\lambda}} |e^{-u} - e^{-u_\lambda}| \cdot |\nabla u| \cdot |\nabla w_\lambda^+|  \ \psi_\varepsilon^2 \phi_{\tau}^2 \, dx\\
		& \leq \mathcal{L}_{\text{exp}}  \int_{\Omega_\lambda} |\nabla u| \cdot |\nabla w_\lambda^+| \ w_\lambda^+ \psi_\varepsilon^2 \phi_{\tau}^2 \, dx\\
		& \leq \mathcal{C}_1 \int_{\Omega_\lambda} |\nabla w_\lambda^+| \  w_\lambda^+ \psi_\varepsilon^2 \phi_{\tau}^2 \, dx,
	\end{split}
\end{equation}
where $\mathcal{C}_1:=\mathcal{L}_{\text{exp}}  \cdot \| \nabla u\|_{L^\infty(\Omega_\lambda)}$ and $\mathcal{L}_{\text{exp}}$ is the Lipschitz constant of the function $t \mapsto e^{-t}$. Hence,  applying the weighted Young's inequality  in \eqref{I_1} we obtain
\begin{equation}\label{I_1bis}
	\begin{split}
		\mathcal{I}_1 &  \leq \mathcal{C}_1 \int_{\Omega_\lambda}  |\nabla w_\lambda^+| \ \psi_\varepsilon \phi_{\tau} \cdot   w_\lambda^+ \psi_\varepsilon \phi_{\tau} \, dx\\
		& \leq \delta_1 \int_{\Omega_\lambda} |\nabla w_\lambda|^2 \psi_\varepsilon^2 \phi_{\tau}^2 \, dx +\mathcal{C}_{\delta_1} \int_{\Omega_\lambda} (w_\lambda^+)^2 \psi_\varepsilon^2 \phi_{\tau}^2 \, dx,
	\end{split}
\end{equation}
where $\delta_1>0$ and $\mathcal{C}_{\delta_1}:=\frac{\mathcal{C}_1^2}{4 \delta_1}$.

In the second term we apply again the weighted Young's inequality in order to get

\begin{equation} \label{I_2}
	\begin{split}
		\mathcal{I}_2 &= 2 \int_{\Omega_{\lambda}} e^{-u_\lambda} |\nabla w_\lambda^+| \cdot  |\nabla \psi_\varepsilon| \ w_\lambda^+ \psi_\varepsilon \phi_{\tau}^2 \, dx\\
		&= 2 \int_{\Omega_{\lambda}} e^{-\frac{u_\lambda}{2}} |\nabla w_\lambda^+| \  \psi_\varepsilon \phi_{\tau} \cdot  e^{-\frac{u_\lambda}{2}} |\nabla \psi_\varepsilon| \  w_\lambda^+ \phi_{\tau} \, dx\\
		&\leq \delta_2 \int_{\Omega_{\lambda}} e^{-u_\lambda} |\nabla w_\lambda^+|^2  \psi_\varepsilon^2 \phi_{\tau}^2 \, dx + \mathcal{C}_{\delta_2}\int_{\Omega_{\lambda}}  e^{-u_\lambda} |\nabla \psi_\varepsilon|^2 \ (w_\lambda^+)^2 \phi_{\tau}^2 \, dx,\\
	\end{split}
\end{equation}
where $\delta_2>0$ and $\mathcal{C}_{\delta_2}:=\frac{\mathcal{C}_1^2}{\delta_2}$.

Repeating the same arguments used for $\mathcal{I}_2$, we deduce an analogous estimation for $\mathcal{I}_3$, i.e.

\begin{equation} \label{I_3}
	\begin{split}
		\mathcal{I}_3 \leq \delta_3 \int_{\Omega_{\lambda}} e^{-u_\lambda} |\nabla w_\lambda^+|^2  \psi_\varepsilon^2 \phi_{\tau}^2 \, dx + \mathcal{C}_{\delta_3}\int_{\Omega_{\lambda}}  e^{-u_\lambda} |\nabla \phi_\tau|^2 \ (w_\lambda^+)^2 \psi_\varepsilon^2 \, dx,\\
	\end{split}
\end{equation}
where $\delta_3>0$ and $\mathcal{C}_{\delta_3}:=\frac{\mathcal{C}_1^2}{\delta_3}$.

In the fourth term, once we use that $|\nabla u| \in L^\infty(\Omega_\lambda)$, we obtain

\begin{equation} \label{I_4}
\begin{split}
	\mathcal{I}_4&=2 \int_{\Omega_{\lambda}} |e^{-u} - e^{-u_\lambda}| \cdot  |\nabla u| \cdot  |\nabla \psi_\varepsilon| \ w_\lambda^+ \psi_\varepsilon \phi_{\tau}^2 \, dx\\
	&\leq 2\mathcal{C}_1 \int_{\Omega_{\lambda}}  |\nabla \psi_\varepsilon| \ w_\lambda^+ \phi_{\tau} \cdot w_\lambda^+ \psi_\varepsilon \phi_{\tau} \, dx,
\end{split}
\end{equation}
where we recall that $\mathcal{C}_1=\mathcal{L}_{\text{exp}}  \cdot \| \nabla u\|_{L^\infty(\Omega_\lambda)}$. Using the classical Young's inequality in \eqref{I_4} we deduce
\begin{equation}\label{I_4bis}
	\begin{split}
		\mathcal{I}_4 \leq  \int_{\Omega_\lambda}|\nabla \psi_\varepsilon|^2 (w_\lambda^+)^2		\phi_{\tau}^2
		\,dx + \mathcal{C}_1^2 \int_{\Omega_\lambda}(w_\lambda^+)^2 \psi_\varepsilon^2
		\phi_{\tau}^2\,dx.
	\end{split}
\end{equation}
Repeating the same arguments used for $\mathcal{I}_4$, we deduce an analogous estimation for $\mathcal{I}_5$, i.e.
\begin{equation}\label{I_5}
	\begin{split}
		\mathcal{I}_5 \leq  \int_{\Omega_\lambda}|\nabla \phi_{\tau}|^2 	(w_\lambda^+)^2	 \psi_\varepsilon^2 \,dx + \mathcal{C}_1^2 \int_{\Omega_\lambda}(w_\lambda^+)^2 \psi_\varepsilon^2	\phi_{\tau}^2\,dx.
	\end{split}
\end{equation}

Finally, it is easy to see that also the function $g(x,t):= e^{-t} f(x, t)$ fulfills $(\mathbf{hp}_f)$. Hence, using our assumption $(\mathbf{hp}_f)$ on the nonlinearity $g$, i.e. that $g(\cdot, s)$ is non-decreasing in the $x_1$-direction and  it is uniformly locally Lipschitz continuous far from the singular set, and choosing $\mathcal{K}= \overline{\Omega}_\lambda$ and $\Lambda=\|u\|_{L^\infty(\Omega_\lambda)}$ in Definition \ref{def:nonlinAssump}, we can get the estimate for the last term of  \eqref{stima1}
\begin{equation}\label{I_6}
\begin{split}
\mathcal{I}_6 \leq \int_{\Omega_\lambda} (e^{-u}f(x,u)-e^{-u_\lambda}f(x,u_\lambda)) \ w_\lambda^+ \psi_\varepsilon^2
\phi_{\tau}^2 \,dx
& \leq \mathcal{L}_g \int_{\Omega_\lambda}(w_\lambda^+)^2 \psi_\varepsilon^2
\phi_{\tau}^2\,dx,
\end{split}
\end{equation}
where $\mathcal{L}_g:=\mathcal{L}_g(\Omega_\lambda, \|u\|_{L^\infty(\Omega_\lambda)})$ is the Lipschitz constant of the function $g(x,t)= e^{-t} f(x, t)$.

Collecting \eqref{I_1bis}, \eqref{I_2}, \eqref{I_3}, \eqref{I_4bis}, \eqref{I_5} and \eqref{I_6}, by \eqref{stima1} we deduce that
\begin{equation}\label{stimaFinal}
	\begin{split}
		\int_{\Omega_\lambda} e^{-u_\lambda}|\nabla w_\lambda^+|^2\psi_\varepsilon^2
		\phi_{\tau}^2 \,dx  \leq& \, \delta_1 \int_{\Omega_\lambda} |\nabla w_\lambda|^2 \psi_\varepsilon^2 \phi_{\tau}^2 \, dx +\mathcal{C}_{\delta_1} \int_{\Omega_\lambda} (w_\lambda^+)^2 \psi_\varepsilon^2 \phi_{\tau}^2 \, dx\\
		&+ \delta_2 \int_{\Omega_{\lambda}} e^{-u_\lambda} |\nabla w_\lambda^+|^2  \psi_\varepsilon^2 \phi_{\tau}^2 \, dx + \mathcal{C}_{\delta_2}\int_{\Omega_{\lambda}}  e^{-u_\lambda} |\nabla \psi_\varepsilon|^2 \ (w_\lambda^+)^2 \phi_{\tau}^2 \, dx\\		
		&+\delta_3 \int_{\Omega_{\lambda}} e^{-u_\lambda} |\nabla w_\lambda^+|^2  \psi_\varepsilon^2 \phi_{\tau}^2 \, dx + \mathcal{C}_{\delta_3}\int_{\Omega_{\lambda}}  e^{-u_\lambda} |\nabla \phi_\tau|^2 \ (w_\lambda^+)^2 \psi_\varepsilon^2 \, dx\\
		&+ \int_{\Omega_\lambda}|\nabla \psi_\varepsilon|^2 (w_\lambda^+)^2		\phi_{\tau}^2
		\,dx + \mathcal{C}_1^2 \int_{\Omega_\lambda}(w_\lambda^+)^2 \psi_\varepsilon^2
		\phi_{\tau}^2\,dx	\\
		&+  \int_{\Omega_\lambda}|\nabla \phi_{\tau}|^2 	(w_\lambda^+)^2	 \psi_\varepsilon^2 \,dx + \mathcal{C}_1^2 \int_{\Omega_\lambda}(w_\lambda^+)^2 \psi_\varepsilon^2	\phi_{\tau}^2\,dx\\
		& + \mathcal{L}_g \int_{\Omega_\lambda}(w_\lambda^+)^2 \psi_\varepsilon^2
		\phi_{\tau}^2\,dx.
	\end{split}
\end{equation}
Let us observe that all the integrals in \eqref{stima1} and also in \eqref{stimaFinal} are computed in the support of the function $w_\lambda^+$, i.e. the set
$$\text{supp}(w_\lambda^+):=\{ x \in \Omega_{\lambda} \setminus R_\lambda (\Gamma) : u(x) > u_\lambda(x) >0 \} \subset {\overline\Omega}_\lambda \subset \overline{\Omega} \setminus \Gamma$$ 
and, hence, we can define $\Omega_\lambda^+ := \Omega_\lambda \cap \text{supp}(w_\lambda^+) \subset \overline{\Omega}_\lambda \subset \Omega \setminus \Gamma$. Having in mind that actually we are working in $\Omega_\lambda^+$, we remark that there exists a positive constant $\mathcal{C}_u$ such that $e^{-u_\lambda} \geq e^{-u} \geq \mathcal{C}_u>0$. Moreover, since $u_\lambda>0$ then $e^{-u_\lambda} <1$, and $\|w_\lambda^+\|_{L^\infty(\Omega_\lambda)} \leq \|u\|_{L^\infty(\Omega_\lambda)}$, we have
\begin{equation}\label{stimaFinalBis}
	\begin{split}
		\mathcal{C}_u \int_{\Omega_\lambda} |\nabla w_\lambda^+|^2\psi_\varepsilon^2
		\phi_{\tau}^2 \,dx  \leq& \, \delta \int_{\Omega_\lambda} |\nabla w_\lambda|^2 \psi_\varepsilon^2 \phi_{\tau}^2 \, dx + \mathcal{A}_{\delta_2} \int_{\Omega_\lambda}|\nabla \psi_\varepsilon|^2 
		\phi_{\tau}^2
		\,dx \\
		&+ \mathcal{A}_{\delta_3} \int_{\Omega_\lambda}|\nabla \phi_{\tau}|^2 	\psi_\varepsilon^2 \,dx +  \mathcal{\bar C} \int_{\Omega_\lambda}(w_\lambda^+)^2 \psi_\varepsilon^2
		\phi_{\tau}^2\,dx,
	\end{split}
\end{equation}
where $\delta:=\delta_1+\delta_2+\delta_3$, $\mathcal{A}_{\delta_2}:= \|u\|_{L^\infty(\Omega_\lambda)}^2(1+\mathcal{C}_{\delta_2})$, $\mathcal{A}_{\delta_3}:=\|u\|_{L^\infty(\Omega_\lambda)}^2(1+\mathcal{C}_{\delta_3})$,  $\mathcal{\bar C}:=\mathcal{C}_{\delta_1} + 2 \mathcal{C}_1^2 + \mathcal{L}_g$. Fixing $\delta>0$ sufficiently small such that $\mathcal{C}_u - \delta>0$ and recalling that $0 \leq \psi_\varepsilon \leq 1$ and $0 \leq \phi_\tau \leq 1$, we have
\begin{equation}\label{stimaFinalTris}
	\begin{split}
		\int_{\Omega_\lambda} |\nabla w_\lambda^+|^2\psi_\varepsilon^2
		\phi_{\tau}^2 \,dx  \leq& \, \frac{\mathcal{A}_{\delta}}{\mathcal{C}_u - \delta} \left( \int_{\Omega_\lambda}|\nabla \psi_\varepsilon|^2 \,dx 
		+ \int_{\Omega_\lambda}|\nabla \phi_{\tau}|^2  \,dx \right)\\
		& +  \frac{\mathcal{\bar C}}{\mathcal{C}_u - \delta} \int_{\Omega_\lambda}(w_\lambda^+)^2 \psi_\varepsilon^2
		\phi_{\tau}^2\,dx,
	\end{split}
\end{equation}
where $\mathcal{A}_{\delta}:=\max\{\mathcal{A}_{\delta_2}, \mathcal{A}_{\delta_3}\}$.
Taking into account the properties of
$\psi_\varepsilon$ and $\phi_\tau$, we have that
\begin{equation}\label{bbeg1}
\int_{\Omega_\lambda}|\nabla \psi_\varepsilon|^2\,dx=
\int_{\Omega_\lambda\cap(\mathcal B_\varepsilon^\lambda\setminus
\mathcal B_{\delta}^\lambda)}|\nabla \psi_\varepsilon|^2\,dx < \Upsilon_1
\varepsilon,
\end{equation}
\begin{equation}\label{bbeg2}
\int_{\Omega_\lambda}|\nabla \phi_\tau|^2\,dx=
\int_{\Omega_\lambda\cap(\mathcal I^\lambda_\tau \setminus \mathcal
I^\lambda_{\sigma})}|\nabla \phi_\tau|^2\,dx < \Upsilon_2 \tau,
\end{equation}
which combined with $0\leq w_\lambda^+\leq u$,  immediately lead to
\begin{equation}\nonumber
\begin{split}
\int_{\Omega_\lambda}|\nabla w_\lambda^+|^2\psi_\varepsilon^2
\phi_{\tau}^2 \,dx \leq \frac{\mathcal{A}_{\delta} \Upsilon}{\mathcal{C}_u-\delta}(\varepsilon +\tau)  + \frac{\mathcal{\bar C}}{\mathcal{C}_u - \delta} \|u\|^2_{L^\infty
(\Omega_\lambda)} \vert \Omega \vert \,,
\end{split}
\end{equation}
where $\Upsilon:=\max\{\Upsilon_1, \Upsilon_2\}$.
By Fatou's Lemma, as $\varepsilon$ and $\tau$ tend to zero, we deduce \eqref{eq:H10test} with $\displaystyle \mathbf{C}( f, |\Omega|, \|u\|_{L^\infty(\Omega_\lambda)}):= \frac{\mathcal{\bar C}}{\mathcal{C}_u - \delta} \|u\|^2_{L^\infty
	(\Omega_\lambda)} \vert \Omega \vert$.

To deduce that $w_\lambda^+ \in H^1_0(\Omega_{\lambda})$, we just note that $w_\lambda^+$ is bounded and then apply standard arguments. See
e.g. Chapter 4 in \cite{evans} having in mind our assumption on the critical set $\Gamma$.

\end{proof}

Now we are ready to prove the main result of this paper.

\begin{proof}[Proof of Theorem \ref{main}]

Let
\begin{equation}\nonumber
\Lambda_0=\{\mathbf{a}<\lambda<0 : u\leq
u_{t}\,\,\,\text{in}\,\,\,\Omega_t\setminus
R_t(\Gamma)\,\,\,\text{for all $t\in(\mathbf{a},\lambda]$}\}.
\end{equation}
In order to apply the moving plane procedure we need to prove some crucial facts. For simplicity of exposition, we divide the proof into three steps.

\textbf{Step 1: we claim that $\Lambda_0 \neq \emptyset$.} Fix  $ \lambda_0 \in (\mathbf{a},0)$ such that
$R_{\lambda_0}(\Gamma) \subset \Omega^c:=\R^n \setminus \Omega$, then for every $\mathbf{a}< \lambda < \lambda_0$, we also have that
$R_\lambda(\Gamma)\subset \Omega^c$.
For any $ \lambda$ in this set we consider (on the domain $\Omega$) the functions $\varphi_{1,\lambda}\,:=\, e^{-u} w_\lambda^+ \phi_{\tau}^2  \chi_{\Omega_\lambda}$ and $\varphi_{2,\lambda}\,:=\, e^{-u_\lambda} w_\lambda^+ \phi_{\tau}^2  \chi_{\Omega_\lambda},$ where $\phi_{\tau}$ is as in \eqref{test2}. We proceed as in the proof of Lemma \ref{leaiuto2}, that is, by Lemma \ref{leaiuto} and a density argument, we can use $\varphi_{1,\lambda}$ and $\varphi_{2,\lambda}$ as test functions in
\eqref{debil1} and \eqref{debil2} so that, subtracting and repeating verbatim the computation done from \eqref{eq:substtest1}  to \eqref{stimaFinalTris}, but excluding the cut-off function $\psi_\varepsilon$, we get
\begin{equation}\nonumber
	\int_{\Omega_\lambda} |\nabla w_\lambda^+|^2
	\phi_{\tau}^2 \,dx  \leq \, \frac{\mathcal{A}_{\delta_3}}{\mathcal{C}_u - \delta}   \int_{\Omega_\lambda}|\nabla \phi_{\tau}|^2  \,dx  +  \frac{\mathcal{\bar C}}{\mathcal{C}_u - \delta} \int_{\Omega_\lambda}(w_\lambda^+)^2 
	\phi_{\tau}^2\,dx,
\end{equation}
where we fixed $\delta=\delta_1+\delta_3>0$ such that $\mathcal{C}_u -\delta>0$, $\mathcal{A}_{\delta_3}:= \|u\|_{L^\infty(\Omega_{\lambda_0})}^2(1+\mathcal{C}_{\delta_3})$,  $\mathcal{\bar C}:=\mathcal{C}_{\delta_1} +  \mathcal{C}_1^2 + \mathcal{L}_g$, $\mathcal{C}_1:=\mathcal{L}_{\text{exp}}  \cdot \| \nabla u\|_{L^\infty(\Omega_{\lambda_0})}$, $\mathcal{C}_{\delta_1}:= \frac{\mathcal{C}_1^2}{4 \delta_1}$, $\mathcal{C}_{\delta_3}:=\frac{\mathcal{C}_1^2}{\delta_3}$
and  $\mathcal{L}_g:=\mathcal{L}_g(\overline{\Omega}_{\lambda_0}, \|u\|_{L^\infty(\Omega_{\lambda_0})})$ is the Lipschitz constant of the function $g(x,t)= e^{-t} f(x, t)$, with  $\mathcal{K}={\overline{\Omega}_{\lambda_0}}$ and $\Lambda= \|u\|^2_{L^\infty
(\Omega_{\lambda_0})}$.

Taking into account the properties of $\phi_\tau$, we deduce that
\begin{equation}\nonumber
\begin{split}
\int_{\Omega_\lambda}|\nabla w_\lambda^+|^2 \phi_{\tau}^2 \,dx \leq
\frac{\mathcal{A}_{\delta_3}\Upsilon_1}{\mathcal{C}_u - \delta} \cdot \tau +\frac{\mathcal{\bar C}}{\mathcal{C}_u - \delta} \int_{\Omega_\lambda}(w_\lambda^+)^2 
\phi_{\tau}^2\,dx.
\end{split}
\end{equation}
By Fatou's Lemma, as $\tau$ tend to, zero we have
\begin{equation}\label{ff}
\begin{split}
\int_{\Omega_\lambda}|\nabla w_\lambda^+|^2  \,dx &\leq
\frac{\mathcal{\bar C}}{\mathcal{C}_u - \delta} \int_{\Omega_\lambda}(w_\lambda^+)^2 \, dx \\
& \leq \frac{\mathcal{\bar C}}{\mathcal{C}_u - \delta} \cdot C_p^2(\Omega_\lambda)
\int_{\Omega_\lambda}|\nabla w_\lambda^+|^2  \,dx,
  \end{split}
\end{equation}
where $C_p(\cdot)$ is the Poincar\'e constant (in the Poincar\'e inequality in $H^1_0(\Omega_\lambda)$).
Since $C_p^2(\Omega_\lambda) \to 0$ as $ \lambda \to \mathsf{a}$, we can find $ \lambda_1 \in (\mathsf{a}, \lambda_0)$, such that
$$\forall \lambda \in (\mathsf{a}, \lambda_1) \qquad \frac{\mathcal{\bar C}}{\mathcal{C}_u - \delta} \cdot C_p^2(\Omega_\lambda)< \frac{1}{2}\,, $$ 
and hence by \eqref{ff}, we deduce that
$$ \forall \lambda \in (\mathbf{a}, \lambda_1) \qquad \int_{\Omega_\lambda}|\nabla w_\lambda^+|^2  \,dx \leq 0,$$ 
proving that $u \leq u_\lambda$ in $\Omega_\lambda \setminus R_\lambda(\Gamma)$ for $\lambda$ close to $\mathbf{a}$, which implies the desired conclusion  $\Lambda_0 \neq \emptyset$.

\noindent Now we can set
 \begin{equation}\nonumber
\lambda_0=\sup\,\Lambda_0.
\end{equation}

\textbf{Step 2: we claim that $\lambda_0=0$.} To this end we assume that
$\lambda_0<0$ and we reach a contradiction by proving that $u\leq
u_{\lambda_0+\nu}$ in $\Omega_{\lambda_0+\nu}\setminus
R_{\lambda_0+\nu}(\Gamma)$ for any $0<\nu<\bar\nu$ for some small
$\bar\nu>0$. By continuity, we know that $u\leq u_{\lambda_0}$ in
$\Omega_{\lambda_0}\setminus R_{\lambda_0}(\Gamma)$.
Since $\Omega$ is convex in the $x_1-$direction and the set $ R_{\lambda_0}(\Gamma)$ lies in the hyperplane $\Pi_{- 2 \lambda_0}$, we see that $\Omega_{\lambda_0}\setminus R_{\lambda_0}(\Gamma)$ is open and connected. Therefore, by the strong maximum principle we deduce that
$u< u_{\lambda_0}$ in $\Omega_{\lambda_0}\setminus R_{\lambda_0}(\Gamma)$ (here we have used that
$ u, u_{\lambda_0}\in C^1(\Omega_{\lambda_0}\setminus R_{\lambda_0}(\Gamma)) $ by standard elliptic estimate). 

Now, we note that for any compact set $\mathcal{X} \subset \Omega_{\lambda_0}\setminus R_{\lambda_0}(\Gamma)$, there exists $\nu=\nu(\mathcal{X},\lambda_0)>0$ sufficiently small such that $\mathcal{X} \subset \Omega_{\lambda} \setminus R_{\lambda}(\Gamma)$ for every $ \lambda \in [\lambda_0, \lambda_0 + \nu].$ Consequently, $u$ and $u_{\lambda}$ are well defined on $\mathcal{X}$ for every $ \lambda \in [\lambda_0, \lambda_0 + \nu].$ Hence, by the uniform continuity of the function $ h(x,\lambda) := u(x) - u(2\lambda-x_1,x') $ on the compact set $\mathcal{X} \times [\lambda_0, \lambda_0 + \nu]$ we can ensure that  $\mathcal{X} \subset \Omega_{\lambda_0 + \nu} \setminus R_{\lambda_0 + \nu}(\Gamma)$ and
$u< u_{\lambda_0+\nu}$ in $\mathcal{X}$ for any $0 \le \nu<\bar\nu$, for some $\bar\nu = \bar\nu(\mathcal{X},\lambda_0)>0$ small. Clearly, we can also assume that $\bar\nu < \frac{\vert \lambda_0 \vert}{4}. $

Let us consider $\psi_\varepsilon$ constructed
in such a way that it vanishes in a neighborhood  of $R_{\lambda_0 +
\nu}(\Gamma)$ and $\phi_{\tau}$ constructed in such a way it
becomes zero in a neighborhood  of $\gamma_{\lambda_0+\nu}=\partial
\Omega \cap T_{\lambda_0+ \nu}$. Thanks to Lemma \ref{leaiuto2}, the functions
$$\varphi_{1, \lambda_0+\nu}\,:= \begin{cases}
	\, \, e^{-u} w_{\lambda_0+\nu}^+\psi_\varepsilon^2 \phi_{\tau}^2 \, & \text{in}\quad {\Omega_{\lambda_0+\nu}} \\
	0 &  \text{in}\quad {\R^n \setminus \Omega_{\lambda_0+\nu}}
\end{cases}$$
and
$$\varphi_{2,\lambda_0+\nu}\,:= \begin{cases}\, e^{-u_{\lambda_0+\nu}} w_{\lambda_0+\nu}^+ \psi_\varepsilon^2 \phi_{\tau}^2  & \text{in}\quad\Omega_{\lambda_0+\nu} \\
0 &  \text{in}\quad \R^n\setminus \Omega_{\lambda_0+\nu}
\end{cases}$$

are such that $ \varphi_{1, \lambda_0+\nu} \to  w_{\lambda_0+\nu}^+$,  $\varphi_{2, \lambda_0+\nu} \to  w_{\lambda_0+\nu}^+$ in $H^1_0(\Omega_{\lambda_0+\nu})$,  as $\varepsilon$ and $\tau$ tend to zero. Moreover, $\varphi_{1, \lambda_0+\nu}, \varphi_{2, \lambda_0+\nu} \in C^{0,1}(\overline{\Omega}_{\lambda_0+\nu})$ and ${\varphi_{1, \lambda_0+\nu}}_{\vert _{\partial \Omega_{\lambda_0+\nu}}} = {\varphi_{2, \lambda_0+\nu}}_{\vert _{\partial \Omega_{\lambda_0+\nu}}}=0$, by Lemma \ref{leaiuto}, and ${\varphi_{1, \lambda_0+\nu}}_{\vert _{\partial \Omega_{\lambda_0+\nu}}} = {\varphi_{2, \lambda_0+\nu}}_{\vert _{\partial \Omega_{\lambda_0+\nu}}}=0$ on an open neighborhood of $\mathcal{X}$, by the above argument.
Therefore, $ \varphi_{1, \lambda_0+\nu}, \varphi_{2, \lambda_0+\nu} \in H^1_0 (\Omega_{\lambda_0+\nu} \setminus \mathcal{X})$ and thus, also $w_{\lambda_0+\nu}^+$ belongs to  $H^1_0(\Omega_{\lambda_0+\nu} \setminus \mathcal{X})$. We also note that $ \nabla w_{\lambda_0+\nu}^+ = 0$ on an open neighborhood of $\mathcal{X}$.

\noindent Now we argue as in Lemma \ref{leaiuto2} and we plug $\varphi_{1, \lambda_0+\nu}, \varphi_{2, \lambda_0+\nu}$ as test functions respectively in \eqref{debil1} and \eqref{debil2} so that, by subtracting and repeating verbatim the computations from \eqref{eq:substtest1}  to \eqref{stimaFinalTris}, we get 
\begin{equation}\label{eq:stimaFinaleContrMoving}
	\begin{split}
		\int_{\Omega_{\lambda_0+\nu}} |\nabla w_{\lambda_0+\nu}^+|^2\psi_\varepsilon^2
		\phi_{\tau}^2 \,dx  \leq& \, \frac{\mathcal{A}_\delta}{\mathcal{C}_u - \delta} \left( \int_{\Omega_{\lambda_0+\nu}}|\nabla \psi_\varepsilon|^2 \,dx 
		+ \int_{\Omega_{\lambda_0+\nu}}|\nabla \phi_{\tau}|^2  \,dx \right)\\
		& +  \frac{\mathcal{\bar C}}{\mathcal{C}_u - \delta} \int_{\Omega_{\lambda_0+\nu}}(w_{\lambda_0+\nu}^+)^2 \psi_\varepsilon^2
		\phi_{\tau}^2\,dx,
	\end{split}
\end{equation}
where we fixed $\delta:=\delta_1+ \delta_2 + \delta_3>0$ such that $\mathcal{C}_u -\delta>0$,  $\mathcal{\bar C}:=\mathcal{C}_{\delta_1} +  2\mathcal{C}_1^2 + \mathcal{L}_g$, $\mathcal{A}_\delta := \max\{\mathcal{A}_{\delta_2}, \mathcal{A}_{\delta_3}\}$, $\mathcal{A}_{\delta_2}:= \|u\|_{L^\infty(\Omega_{\lambda_0+\nu})}^2(1+\mathcal{C}_{\delta_2})$, $\mathcal{A}_{\delta_3}:=\|u\|_{L^\infty(\Omega_{\lambda_0+\nu})}^2(1+\mathcal{C}_{\delta_3})$, $\mathcal{C}_1:=\mathcal{L}_{\text{exp}}  \cdot \| \nabla u\|_{L^\infty(\Omega_{\lambda_0+\nu})}$, $\mathcal{C}_{\delta_1}:= \frac{\mathcal{C}_1^2}{4 \delta_1}$, $\mathcal{C}_{\delta_2}:= \frac{\mathcal{C}_1^2}{\delta_2}$, $\mathcal{C}_{\delta_3}:= \frac{\mathcal{C}_1^2}{\delta_3}$ and  $\mathcal{L}_g:=\mathcal{L}_g(\overline{\Omega}_{\lambda_0+\nu}, \|u\|_{L^\infty(\Omega_{\lambda_0+\nu})})$ is the Lipschitz constant of the function $g(x,t)= e^{-t} f(x, t)$, with  $\mathcal{K}={\overline{\Omega}_{\lambda_0+\nu} \setminus R_{\lambda_0 + \nu}(\Gamma)}$ and $\Lambda= \|u\|^2_{L^\infty(\Omega_{\lambda_0+\nu})}$.

Passing to the limit, as $(\epsilon, \tau) \to (0,0),$ in the latter
we get

\begin{equation}\label{sdjfhsfskl}
\begin{split}
\int_{\Omega_{\lambda_0+\nu}\setminus \mathcal{X}}|\nabla
w_{\lambda_0+\nu}^+|^2\,dx &\leq \frac{\mathcal{\bar C}}{\mathcal{C}_u - \delta} 
\int_{\Omega_{\lambda_0+\nu}\setminus \mathcal{X}}  (w_{\lambda_0+\nu}^+)^2\,dx\\
&\leq \frac{\mathcal{\bar C}}{\mathcal{C}_u - \delta} \cdot
C_p^2(\Omega_{\lambda_0+\nu}\setminus \mathcal{X})
\int_{\Omega_{\lambda_0+\nu}\setminus \mathcal{X}}   |\nabla
w_{\lambda_0+\nu}^+|^2\,dx\, ,
  \end{split}
\end{equation}
where $C_p(\cdot)$ is the Poincar\'e constant (of the Poincar\'e inequality in $H^1_0(\Omega_{\lambda_0+\nu}\setminus \mathcal{X})$).
{Now we recall that $C_p^2(\Omega_{\lambda_0+\nu}\setminus \mathcal{X}) \le \mathcal{Q}(n) \vert \Omega_{\lambda_0+\nu}\setminus \mathcal{X} \vert ^{\frac{2}{n}} $, where $\mathcal{Q}=\mathcal{Q}(n)$ is a positive constant depending only on the dimension $n$, and therefore, by summarizing, we have proved that for every compact set $\mathcal{X} \subset \Omega_{\lambda_0}\setminus R_{\lambda_0}(\Gamma)$ there is a small $\bar\nu = \bar\nu(\mathcal{X}, \lambda_0) \in \left(0, \frac{\vert \lambda_0 \vert}{4}\right)$ such that for every $ 0 \le \nu < \bar\nu$ we have
\begin{equation}\label{sdjfhsfskl2}
\begin{split}
&\int_{\Omega_{\lambda_0+\nu}\setminus \mathcal{X}}|\nabla
w_{\lambda_0+\nu}^+|^2\,dx \leq \frac{\mathcal{\bar C}}{\mathcal{C}_u - \delta} \cdot \mathcal{Q}(n) \vert \Omega_{\lambda_0+\nu}\setminus \mathcal{X} \vert ^{\frac{2}{n}}
\int_{\Omega_{\lambda_0+\nu}\setminus \mathcal{X}}   |\nabla
w_{\lambda_0+\nu}^+|^2\,dx.
\end{split}
\end{equation}
Now we first fix a compact $\mathcal{X} \subset \Omega_{\lambda_0}\setminus R_{\lambda_0}(\Gamma)$ such that
$$\vert \Omega_{\lambda_0}\setminus \mathcal{X} \vert^{\frac{2}{n}} < \left[ \frac{2\mathcal{\bar C}}{\mathcal{C}_u - \delta} \cdot \mathcal{Q}(n) \right]^{-1},$$
this is possible since $ \vert  R_{\lambda_0}(\Gamma)\vert  =0$ by the assumption on $ \Gamma$, and then we take $ \bar\nu_0 < \bar\nu $ such that for every $ 0 \le \nu < \bar\nu_0$ we have
$\vert \Omega_{\lambda_0 + \nu} \setminus \Omega_{\lambda_0}
\vert^{\frac{2}{n}}< \left[ \frac{2\mathcal{\bar C}}{\mathcal{C}_u - \delta} \cdot \mathcal{Q}(n) \right]^{-1}$. Inserting those
information into \eqref{sdjfhsfskl2} we immediately get that
\begin{equation}\label{sdjfhsfskl3}
\begin{split}
&\int_{\Omega_{\lambda_0+\nu}\setminus \mathcal{X}}|\nabla
w_{\lambda_0+\nu}^+|^2\,dx  < \frac{1}{2}
\int_{\Omega_{\lambda_0+\nu}\setminus \mathcal{X}}   |\nabla
w_{\lambda_0+\nu}^+|^2\,dx
\end{split}
\end{equation}
and so $\nabla w_{\lambda_0+\nu}^+=0$ on $ \Omega_{\lambda_0+\nu} \setminus \mathcal{X}$ for every $ 0 \le \nu < \bar\nu_0$. On the other hand, we recall that $\nabla w_{\lambda_0+\nu}^+=0$ on an open neighbourhood of $\mathcal{X}$ for every $ 0 \le \nu < \bar\nu$, thus $\nabla w_{\lambda_0+\nu}^+=0$ on $ \Omega_{\lambda_0+\nu} $ for every $ 0 \le \nu < \bar\nu_0$. The latter proves that $u\leq u_{\lambda_0+\nu}$ in
$\Omega_{\lambda_0+\nu}\setminus R_{\lambda_0+\nu}(\Gamma)$ for every
$0<\nu<\bar\nu_0$. Such a contradiction shows that
\[
\lambda_0=0\,.
\]
}

\textbf{Step 3: conclusion.} Since the moving plane procedure can be
performed in the same way but in the opposite direction, this
proves the desired symmetry result.  The fact that the solution is
increasing in the $x_1$-direction in $\{x_1<0\}$ is implicit in the
moving plane procedure. Since $u$ has $\mathcal{C}^1$ regularity, by standard elliptic estimate, the fact that $\frac{\partial u}{\partial x_1}$ is positive for $x_1<0$
follows by the classical maximum principle.

\end{proof}

\end{document}